\documentclass[11pt,reqno]{amsart}

\usepackage{amsfonts,amsmath,amssymb,amsthm}
\usepackage{anysize}
\usepackage{mathtools}
\marginsize{2cm}{2cm}{1.5cm}{1.5cm}
\usepackage[utf8]{inputenc}
\usepackage[english]{babel}
\usepackage{cmap}
\usepackage{enumerate}
\usepackage{xcolor}
\usepackage{hyperref}
\hypersetup{
  colorlinks = true,
  urlcolor = blue,
  linkcolor = blue,
  citecolor = red,
  pdftitle = {A colorful Steinitz theorem with different centers},
  pdfauthor = {Grigory Ivanov},
  pdfsubject = {Combinatorial convexity}
}

\newtheoremstyle{theorem}%
  {\topsep}{\topsep}{\itshape}{}%
  {\bfseries}{.}{5pt plus 1pt minus 1pt}{}
\theoremstyle{theorem}
\newtheorem{thm}{Theorem}[section]
\newtheorem{prp}{Proposition}[section]
\newtheorem{lem}{Lemma}[section]
\newtheorem{cor}{Corollary}[section]
\theoremstyle{definition}
\newtheorem{dfn}{Definition}[section]

\numberwithin{equation}{section}
\newcommand{\Href}[2]{\hyperref[#2]{#1~\ref{#2}}}
\newcommand{\st}{:\;}
\newcommand{\norm}[1]{\left\|#1\right\|}
\newcommand{\enorm}[1]{\left|#1\right|}
\newcommand{\conv}{\mathrm{conv}}
\newcommand{\iprod}[2]{\left\langle#1,#2\right\rangle}
\providecommand{\parenth}[1]{\left(#1\right)}
\providecommand{\braces}[1]{\left\{#1\right\}}
\providecommand{\brackets}[1]{\left[#1\right]}
\def\R{{\mathbb R}}
\def\Sph{{\mathbb S}}
\def\Lin{\mathop{\rm Lin}}
\def\dim{\mathop{\rm dim}}
\def\distan{\mathrm{dist}}
\newcommand{\poscone}[1]{\mathrm{Pos}\ \! {#1}}
\newcommand{\reccone}[1]{\mathrm{recc}\ \! {#1}}
\providecommand{\abs}[1]{\lvert#1\rvert}
\providecommand{\card}[1]{\lvert#1\rvert}
\newcommand{\ball}[1]{\mathbf{B}^{#1}}

\providecommand{\Bd}{\ball{d}}
\providecommand{\Sd}{\Sph^{d-1}}
\providecommand{\smin}{s_{\min}}
\providecommand{\widthu}{\operatorname{w}_u}
\DeclareMathOperator{\ri}{ri}

\title[A colorful Steinitz theorem with different centers]
{A colorful Steinitz theorem with different centers}
\author{Grigory Ivanov}
\email{grimivanov@gmail.com}
\date{}

\begin{document}

\begin{abstract}
We prove a colorful quantitative Steinitz theorem in which the color
classes may have different centers. If the convex hull of each of
$2d$ sets in $\R^d$ contains a translate of the Euclidean unit ball,
then a rainbow convex hull contains a ball of radius
$(3d)^{-2d^2}$ whose center belongs to the convex hull of the given
centers. The main ingredient is an exact result for translates of a
segment, proved by a lifting and a topological colorful Helly theorem.
The same idea also yields a sharp colorful Helly theorem for translated
cones: for $1 \le k \le d-1$, if every rainbow selection from
$2d-k+1$ finite families of convex sets in $\R^d$ has a $k$-dimensional
cone in its intersection, then the intersection of one of the families
contains such a cone.
\end{abstract}

\maketitle

\section{Introduction}

The goal of this paper is to establish a colorful quantitative version of
the following classical theorem of Steinitz \cite{steinitz1913bedingt}.
For a set $X \subset \R^d$, we write $\conv X$ for its convex hull.
\begin{prp}[Steinitz theorem]
\label{prp:classical_steinitz}
Let $S \subset \R^d$, and suppose that the origin belongs to the interior
of $\conv S$.  Then there is a set $Y \subset S$ of at most $2d$ points
such that the origin belongs to the interior of $\conv Y$.
\end{prp}

B\'ar\'any, Katchalski, and Pach obtained the first quantitative version
\cite{barany1982quantitative}.  Let $\Bd$ denote the Euclidean unit ball
in $\R^d$.  We shall use the following polynomial estimate
\cite[Theorem~1]{ivanov2024steinitz}.
\begin{prp}
\label{prp:quantitative_steinitz}
Let $X \subset \R^d$.  If
\[
\Bd \subset \conv X,
\]
then there is $Y \subset X$ such that
\[
 \card{Y} \le 2d
 \qquad\text{and}\qquad
 \frac{1}{6d^2}\Bd \subset \conv Y.
\]
\end{prp}

Throughout the paper, we fix an integer $d \ge 2$.  For a positive
integer $n$, write
\(
 \brackets{n} := \braces{1, \dots, n}.
\)
We write $\Sd$ for the Euclidean unit sphere in $\R^d$.  In the segment statements,
$u \in \Sd$ denotes a unit vector, and we put
\[
 \brackets{-u,u} := \braces{\lambda u \st \abs{\lambda} \le 1}.
\]

Given sets $X_1, \dots, X_m$, regarded as color classes, a
\emph{transversal}, or a \emph{rainbow choice}, is a labeled choice
$x_i \in X_i$ for every $i \in \brackets{m}$. Points with the same
geometric position but different color labels are treated as distinct.

A common-center colorful counterpart was proved by De Loera, La Haye,
Rolnick, and Sober\'{o}n \cite[Lemma~2.6]{de2017quantitative}.
\begin{prp}[Colorful quantitative Steinitz theorem]
\label{prp:colorful_quantitative_steinitz}
Let $X_1, \dots, X_{2d} \subset \R^d$.  Suppose that
\[
 \Bd \subset \conv X_i,
 \qquad i \in \brackets{2d}.
\]
Then there is a transversal $x_i \in X_i$, $i \in \brackets{2d}$,
such that
\[
  \frac{1}{3}d^{-2d-2} \ \Bd
 \subset \conv\braces{x_1, \dots, x_{2d}}.
\]
\end{prp}
The choice of a center is an essential part of polarity arguments in
quantitative convexity  \cite{ivanov2022quantitative,  ivanov_QST_polarity_2025}.  It is therefore natural to try to prove a
colorful quantitative statement by finding a single center that is
simultaneously good for all colors.

This was the approach we first tried for colorful quantitative
Steinitz-type statements, motivated also by colorful quantitative Helly
theorems for volume and diameter
\cite{soberon2016helly,damasdi2021colorful,Ivanov2026Hellynumbers}.
Our attempts to find such a common center were unsuccessful.  This
suggests that the difficulty may be conceptual rather than merely
technical: in a colorful problem, searching for one common good center
may be the wrong starting point.  The present paper instead allows each
color to retain its own natural center.

The following equivalent reformulation of the colorful
Carath\'{e}odory
theorem \cite{barany1982generalization} provides a qualitative model for
this point of view.
\begin{prp}
\label{prp:colorful_caratheodory_centers}
Let $X_1, \dots, X_{d+1} \subset \R^d$, and let
$c_i \in \conv X_i$ for every $i \in \brackets{d+1}$.  Then there is a
transversal $x_i \in X_i$, $i \in \brackets{d+1}$, such that
\[
 \conv\braces{x_1, \dots, x_{d+1}}
 \cap
 \conv\braces{c_1, \dots, c_{d+1}}
 \ne \emptyset.
\]
\end{prp}

In this sense, our first main result is a quantitative Steinitz analogue
of the preceding proposition: the assumptions $c_i \in \conv X_i$ are
strengthened to $c_i + \Bd \subset \conv X_i$, and the common point can
be chosen as the center of a Euclidean ball of explicit radius contained
in the rainbow convex hull.

\begin{thm}[Colorful Steinitz theorem with different centers]
\label{thm:steinitz}
Let $X_1, \dots, X_{2d} \subset \R^d$, and let
$c_1, \dots, c_{2d} \in \R^d$.  Suppose that
\[
 c_i + \Bd \subset \conv X_i,
 \qquad i \in \brackets{2d}.
\]
Then there is a transversal $x_i \in X_i$, $i \in \brackets{2d}$, and
a point
\[
 c \in \conv\braces{c_1, \dots, c_{2d}}
\]
such that
\[
 c +  \parenth{3d}^{-2d^2} \Bd
 \subset \conv\braces{x_1, \dots, x_{2d}}.
\]
\end{thm}

When all the centers $c_i$ coincide, the preceding theorem gives a weaker
radius than \Href{Proposition}{prp:colorful_quantitative_steinitz}; its
novelty is that the centers may be arbitrary.

The proof of this theorem selects all $2d$ points simultaneously.  Its key
ingredient is the following exact statement for translates of a segment.
It is a point-selection counterpart of Sober\'{o}n's colorful Helly theorem
for width in a fixed direction
\cite[Theorem~2.2]{soberon2016helly}.
\begin{thm}[Colorful translates of a segment]
\label{thm:segment}
Let $u \in \Sd$, let $X_1, \dots, X_{2d} \subset \R^d$, and let
$c_1, \dots, c_{2d} \in \R^d$.  Suppose that
\begin{equation}
 c_i + \brackets{-u, u} \subset \conv X_i,
 \qquad i \in \brackets{2d}.
 \label{eq:segment_hypothesis}
\end{equation}
Then there is a transversal $x_i \in X_i$, $i \in \brackets{2d}$, and
a point
\[
 c \in \conv\braces{c_1, \dots, c_{2d}}
\]
such that
\begin{equation}
 c + \brackets{-u, u}
 \subset \conv\braces{x_1, \dots, x_{2d}}.
 \label{eq:segment_conclusion}
\end{equation}
\end{thm}

The same idea also gives a colorful Helly theorem for cones. The
question of forcing a cone in the intersection of a family of convex
sets goes back to Katchalski \cite{katchalski1978helly}, who proved the
case of rays. B\'ar\'any \cite[Theorem~1.2]{barany2024positive}
established the sharp theorem for $k$-dimensional cones. In
\cite{Ivanov2026Colorful_pos_bas}, we obtained its colorful homogeneous
version, with the optimal number $2d-k+1$ of colors. The nonhomogeneous
statement proved there only ensures that every member of one color
class contains a translate of the same cone; the translates need not
have a common apex. The following theorem settles the question left
open at the end of that paper.

We write $\poscone{A}$ for the positive hull of a set $A \subset \R^d$
and $\Lin A$ for its linear span.
A \emph{convex cone} is a set of the form $a + \poscone{A}$; the point
$a$ is an \emph{apex} of the cone, and need not be the origin.
A \emph{rainbow sub-selection} from several families consists of at
most one set from each family; a \emph{rainbow selection} contains
exactly one set from each family. Put
\[
 m(k,d) := \max\braces{d+1, 2(d-k+1)}.
\]

\begin{thm}[Colorful Helly theorem for cones]
\label{thm:colorful_helly_cones}
Fix $k \in \brackets{d-1}$ and put $N = 2d-k+1$. Let
$\mathcal F_1, \dots, \mathcal F_N$ be finite families of convex sets
in $\R^d$. If the intersection of every rainbow sub-selection of at
most $m(k,d)$ sets contains a $k$-dimensional cone, then, for some
$i \in \brackets{N}$, the intersection
\[
 \bigcap_{F \in \mathcal F_i} F
\]
contains a $k$-dimensional cone.
\end{thm}

Both parameters are optimal. The homogeneous examples in
\cite{Ivanov2026Colorful_pos_bas} already require $N$ colors, even
under the full-rainbow hypothesis. To see that $m(k,d)$ cannot be
reduced, take $N$ copies of a sharp monochromatic family for the
cone theorem: a rainbow sub-selection of at most $m(k,d)-1$ sets
uses at most that many distinct constraints. Our proof also gives
an alternative derivation of the homogeneous colorful theorem.

This result is a simple illustration of our general lifting idea.
We encode a segment by its two endpoints in $\R^d \times \R^d$, so
that containment in a convex set becomes a convex constraint in the
lifted space. For the Steinitz theorem, the lift also records a common
translation of the endpoints, with the center constrained to the
convex hull of the given centers. For cones, we prescribe linear
conditions on the difference of the endpoints and let their separation
tend to infinity. This reduces the cone problem to the classical
Colorful Helly theorem in dimension $2d-k$.

In \Href{Section}{sec:motivation}, we deduce the quantitative Steinitz
theorem from the segment theorem. Sections~\ref{sec:idea}--\ref{sec:segment_proof}
are devoted to the latter. The proof of the cone theorem, which is
independent of the topological part of the argument, is given in
\Href{Section}{sec:cones}.

\subsection*{Use of AI tools}
The author used ChatGPT during the preparation of this paper. The overall proof strategy originated with the author, while the colorful-hole lemma (Lemma 4.2), a key step in the argument, was first suggested by ChatGPT Sol. Theorem 1.3 applies the same underlying idea as the proof of Theorem 1.1; the simplified argument presented here was suggested by ChatGPT Astra.   

\section{From segments to a quantitative Steinitz theorem}
\label{sec:motivation}

In this section, we derive \Href{Theorem}{thm:steinitz} from
\Href{Theorem}{thm:segment}.  We regard the former as a point-selection
counterpart of the colorful quantitative Helly theorem for diameter with
$2d$ colors \cite[Theorem~1]{Ivanov2026Hellynumbers}.

One of the main ingredients in the proof of that quantitative Helly theorem is the fixed-direction
 colorful Helly theorem of Sober\'{o}n
\cite[Theorem~2.2]{soberon2016helly}.  Recall that, for a convex set
$K \subset \R^d$, its width in the direction $u$ is
\[
 \widthu(K)
 := \sup\braces{\iprod{u}{x - y} : x, y \in K}.
\]

\begin{prp}
\label{prp:soberon_width}
Let $\mathcal F_1, \dots, \mathcal F_{2d}$ be non-empty finite
families of convex sets in $\R^d$.  Suppose that
\[
 \widthu \! \parenth{F_1 \cap \cdots \cap F_{2d}} \ge 1
\]
for every choice $F_i \in \mathcal F_i$,
$i \in \brackets{2d}$.  Then there is $j \in \brackets{2d}$ such
that
\[
 \widthu \! \parenth{\bigcap_{F \in \mathcal F_j}F} \ge 1.
\]
\end{prp}

It is natural to view \Href{Theorem}{thm:segment} as the point-selection
dual of this proposition.

Our deduction of \Href{Theorem}{thm:steinitz} from
\Href{Theorem}{thm:segment} uses \Href{Proposition}{prp:quantitative_steinitz}.

The second ingredient converts sufficiently many witness directions
into a Euclidean ball.

\begin{lem}
\label{lem:witness}
Let $Q \subset \R^d$ be compact and convex, let $W \subset \Sd$ have
normalized spherical measure at least $\alpha > 0$, and let $\rho > 0$.
Suppose that, for every $v \in W$, there is $c_v \in \R^d$ such that
\[
 c_v + \brackets{-\rho v, \rho v} \subset Q.
\]
Then $Q$ contains a translate of
\[
 \frac{\rho}{d}\parenth{\frac{\alpha}{2d}}^{d-1}\Bd.
\]
If all the points $c_v$ belong to a convex set $C$, the center of this
ball can be chosen in $C$.
\end{lem}
\begin{proof}
Let $\sigma$ denote the standard Haar probability measure on $\Sd$.
We use the spherical-zone estimate
\[
 \sigma\braces{v \in \Sd : \abs{\iprod{v}{e}} \le \omega}
 \le \sqrt{2d}\,\omega,
 \qquad e \in \Sd,\quad 0 \le \omega \le 1,
\]
 see
\cite{boroczky2003covering}.
  Set
\[
 \delta := \frac{\alpha}{2 \sqrt d}.
\]
Choose $u_1 \in W$.  We can then choose $u_2, \dots, u_d \in W$
successively so that
\[
 \distan\parenth{u_j, \Lin\braces{u_1, \dots, u_{j-1}}}
 \ge \delta,
 \qquad j = 2, \dots, d.
\]
Indeed, at every step the linear span of the previously chosen vectors
is contained in a hyperplane.  Its $\delta$--neighborhood in $\Sd$ is
therefore contained in a zone of half-width $\delta$, whose measure is
at most
\[
 \sqrt{2d}\,\delta  < \alpha
 \le \sigma \! \parenth{W}.
\]

Let $U$ be the matrix with columns $u_1, \dots, u_d .$
On the one hand,
\[
 \abs{\det U}
 = \prod_{j=2}^{d}
 \distan\parenth{u_j, \Lin\braces{u_1, \dots, u_{j-1}}}
 \ge \delta^{d-1}.
\]
On the other hand, $\norm{U}_{\mathrm{op}} \le \sqrt d$.  Hence,
\[
 \smin(U)
 \ge \frac{\abs{\det U}}{\norm{U}_{\mathrm{op}}^{d-1}}
 \ge \parenth{\frac{\alpha}{2 d}}^{d-1}.
\]
Choose witnesses $c_j := c_{u_j}$ and put
\[
 c := \frac{1}{d}\sum_{j \in \brackets{d}}c_j.
\]
By convexity,
\[
 c + \frac{\rho}{d}U\brackets{-1, 1}^d
 = \frac{1}{d}\sum_{j \in \brackets{d}}
 \parenth{c_j + \brackets{-\rho u_j, \rho u_j}}
 \subset Q.
\]
Since $U\brackets{-1, 1}^d$ contains $\smin(U)\Bd$, the conclusion
follows.  If every $c_v$ belongs to $C$,
then $c \in C$.
\end{proof}

\begin{proof}[Proof of \Href{Theorem}{thm:steinitz}]
Put
\[
 C := \conv\braces{c_1, \dots, c_{2d}}
 \qquad\text{and}\qquad
 \rho := \frac{1}{6d^2}.
\]
Apply \Href{Proposition}{prp:quantitative_steinitz} to each color.  We
obtain $Y_i \subset X_i$, $i \in \brackets{2d}$, such that
\[
 \card{Y_i} \le 2d
 \qquad\text{and}\qquad
 c_i + \rho\Bd \subset \conv Y_i.
\]
Let
\[
 \mathcal R := Y_1 \times \cdots \times Y_{2d}.
\]
Thus, the elements $R = \parenth{x_1, \dots, x_{2d}}$ of $\mathcal R$
are labeled transversals, and
\[
 \card{\mathcal R} \le M := \parenth{2d}^{2d}.
\]
For $R \in \mathcal R$, put
\[
 P_R := \conv\braces{x_1, \dots, x_{2d}}
\]
and define its set of witness directions by
\[
 W_R := \braces{v \in \Sd : \text{there is } c \in C \text{ such that }
 c + \brackets{-\rho v, \rho v} \subset P_R}.
\]
Each $W_R$ is compact: it is the projection onto $\Sd$ of the compact
set
\[
 \braces{\parenth{c, v} \in C \times \Sd :
 c - \rho v \in P_R,\ c + \rho v \in P_R}.
\]
For every $v \in \Sd$, apply \Href{Theorem}{thm:segment} to the sets
$\rho^{-1}Y_i$ and the centers $\rho^{-1}c_i$, and then rescale.  It
yields $R \in \mathcal R$ and $c \in C$ such that
\[
 c + \brackets{-\rho v, \rho v} \subset P_R.
\]
Consequently,
\[
 \Sd = \bigcup_{R \in \mathcal R}W_R.
\]
By the pigeonhole principle, some $R \in \mathcal R$ satisfies
\[
 \sigma(W_R)
 \ge \frac{1}{\card{\mathcal R}}
 \ge \frac{1}{M}.
\]

Apply \Href{Lemma}{lem:witness} to $Q = P_R$, $W = W_R$, and
$\alpha = M^{-1}$.  We obtain a point $c \in C$ such that
\[
 c + r_d\Bd \subset P_R,
 \qquad
 r_d := \frac{1}{6d^3}
 \parenth{\frac{1}{2d\parenth{2d}^{2d}}}^{d-1}.
\]
The elementary inequalities
\[
 2d\parenth{2d}^{2d} \le \parenth{3d}^{2d+1}
 \qquad\text{and}\qquad
 6d^3 \le \parenth{3d}^{d+1}
\]
give
\[
 r_d \ge \parenth{3d}^{-2d^2}.
\]
Since $R$ is a transversal of the sets $Y_i \subset X_i$, this proves
the theorem.
\end{proof}

\section{Idea of the proof of the segment theorem}
\label{sec:idea}

In view of \Href{Proposition}{prp:soberon_width} and related colorful
quantitative results, $2d$ is the natural number of colors.  Achieving
this number is also the main difficulty in \Href{Theorem}{thm:segment}.
We first observe
that the non-colorful statement and a version with $2d + 1$ colors
follow readily from the following  form of the
Carath\'{e}odory theorem; see
\cite[Theorem~10.3]{barany2021combinatorial}.

\begin{prp}[Very colorful Carath\'{e}odory theorem]
\label{prp:colorful_caratheodory}
Let $A_1, \dots, A_n \subset \R^n$, let
$p \in \conv A_i$ for every $i \in \brackets{n}$, and let
$q \in \R^n$.  Then there are $a_i \in A_i$,
$i \in \brackets{n}$, such that
\[
 p \in \conv\braces{q, a_1, \dots, a_n}.
\]
\end{prp}

We first record a standard non-colorful consequence.

\begin{prp}
\label{prp:segment_caratheodory}
Let $A \subset \R^d$, and let $S \subset \conv A$ be a segment.  Then
there is $A' \subset A$ with $\card{A'} \le 2d$ such that
\[
 S \subset \conv A'.
\]
\end{prp}

\begin{proof}
Write $S = \brackets{a,b}$. If $a = b$, the usual Carath\'{e}odory
theorem gives the result, since $d+1 \le 2d$. Otherwise, put
$q = (a+b)/2$. Apply \Href{Proposition}{prp:colorful_caratheodory} to
$d$ copies of $A$, first with $p = a$ and then with $p = b$, using
$q$ as the additional point in both applications. We obtain subsets
$A_a, A_b \subset A$, each of size at most $d$, such that
\[
 a \in \conv\parenth{A_a \cup \braces{q}},
 \qquad b \in \conv\parenth{A_b \cup \braces{q}}.
\]
Put $P = \conv\parenth{A_a \cup A_b}$. If $q \notin P$, a hyperplane
strictly separating $q$ from $P$ places both $a$ and $b$ strictly on
the side containing $P$, contrary to $q = (a+b)/2$. Thus $q \in P$,
so $S \subset P$, as required.
\end{proof}

For comparison, a standard lifting proves the desired result with
$2d + 1$ colors.

\begin{lem}
\label{lem:segment_2dplus1}
Let $X_1, \dots, X_{2d + 1} \subset \R^d$, and suppose that
\[
 c_i + \brackets{-u, u} \subset \conv X_i,
 \qquad i \in \brackets{2d + 1}.
\]
Then there are $x_i \in X_i$, $i \in \brackets{2d + 1}$, and
$c \in \conv\braces{c_1, \dots, c_{2d + 1}}$ such that
\[
 c + \brackets{-u, u}
 \subset \conv\braces{x_1, \dots, x_{2d + 1}}.
\]
\end{lem}

\begin{proof}
For $i \in \brackets{2d + 1}$, put
\[
 L_i
 :=
 \braces{
  \parenth{\varepsilon x, \varepsilon, x - c_i} :
  x \in X_i,\ \varepsilon \in \braces{-1, 1}
 }
 \subset \R^d \times \R \times \R^d.
\]
Choose finitely supported coefficients
$\alpha_{i,x}^{+}, \alpha_{i,x}^{-} \ge 0$ such that
\[
 \sum_{x \in X_i}\alpha_{i,x}^{+}
 = \sum_{x \in X_i}\alpha_{i,x}^{-} = 1,
 \qquad
 c_i \pm u = \sum_{x \in X_i}\alpha_{i,x}^{\pm}x.
\]
Then
\[
 \parenth{u, 0, 0}
 = \frac{1}{2}\sum_{x \in X_i}\alpha_{i,x}^{+}
   \parenth{x, 1, x - c_i}
 + \frac{1}{2}\sum_{x \in X_i}\alpha_{i,x}^{-}
   \parenth{-x, -1, x - c_i},
\]
so $\parenth{u, 0, 0} \in \conv L_i$.

Apply \Href{Proposition}{prp:colorful_caratheodory} in
$\R^{2d + 1}$, with the origin as the additional point.  There are
$x_i \in X_i$, $\varepsilon_i \in \braces{-1, 1}$,
$i \in \brackets{2d + 1}$, and coefficients
$\lambda_0, \lambda_1, \dots, \lambda_{2d + 1} \ge 0$ such that
\[
 \parenth{u, 0, 0}
 = \sum_{i \in \brackets{2d + 1}}\lambda_i
   \parenth{\varepsilon_i x_i, \varepsilon_i, x_i - c_i},
 \qquad
 \lambda_0 + \sum_{i \in \brackets{2d + 1}}\lambda_i = 1.
\]
Set
\[
 I_+ := \braces{i \in \brackets{2d + 1} : \varepsilon_i = 1},
 \qquad
 I_- := \braces{i \in \brackets{2d + 1} : \varepsilon_i = -1},
\]
and
\[
 s := \sum_{i \in \brackets{2d + 1}}\lambda_i = 1 - \lambda_0.
\]
Since $u \ne 0$, we conclude $s > 0$.  Comparing the three coordinate
blocks gives
\[
 \sum_{i \in \brackets{2d + 1}}\lambda_i\varepsilon_i x_i = u,
 \qquad
 \sum_{i \in \brackets{2d + 1}}\lambda_i\varepsilon_i = 0,
 \qquad
 \sum_{i \in \brackets{2d + 1}}\lambda_i(x_i - c_i) = 0.
\]
Consequently,
\[
 \sum_{i \in I_+}\lambda_i
 = \sum_{i \in I_-}\lambda_i = \frac{s}{2}.
\]
Therefore,
\[
 p_+ := \frac{2}{s}\sum_{i \in I_+}\lambda_i x_i,
 \qquad
 p_- := \frac{2}{s}\sum_{i \in I_-}\lambda_i x_i
\]
belong to $\conv\braces{x_1, \dots, x_{2d + 1}}$ and satisfy
\[
 p_+ - p_- = \frac{2}{s}u.
\]
Their midpoint is
\[
 c := \frac{p_+ + p_-}{2}
 = \frac{1}{s}\sum_{i \in \brackets{2d + 1}}\lambda_i x_i
 = \frac{1}{s}\sum_{i \in \brackets{2d + 1}}\lambda_i c_i
 \in \conv\braces{c_1, \dots, c_{2d + 1}}.
\]
Thus, $p_\pm = c \pm \frac{1}{s}u$.  Since $s \le 1$, the segment
$\brackets{p_-, p_+}$ contains $c + \brackets{-u, u}$.
\end{proof}

Thus, the essential difficulty is the passage from $2d + 1$ colors to
$2d$ colors.  We next describe the mechanism that makes this passage
possible.  Put
\[
 C := \conv\braces{c_1, \dots, c_{2d}}.
\]
In \Href{Section}{sec:encoding}, we associate with every
$x \in \R^d$ an open convex set $F_x^C$ in a
$2d$--dimensional affine space $E_u$ so that, for every non-empty
finite set $I \subset \R^d$,
\[
 \text{there is }c \in C\text{ with }
 c + \brackets{-u, u} \subset \conv I
 \quad\Longleftrightarrow\quad
 \bigcap_{x \in I}F_x^C = \emptyset.
\]
This is the geometric dualization used in the proof.

The dimension of $E_u$ is still $2d$.  Consequently, an ordinary
colorful Helly-type result would require $2d + 1$ colors.  The particular
sets $F_x^C$, however, have an additional ``escape'' property: for every
finite $I$, every connected component of
\[
 E_u \setminus \bigcup_{x \in I}F_x^C
\]
is unbounded.  After passing to the one-point compactification of
$E_u$, the closures of all these components meet at the point at
infinity, so the compactified complement is connected.  Alexander
duality then forces
the additional vanishing of reduced homology in degree $2d - 1$.
Thus,   the Topological Colorful Helly
theorem \cite{kalai2005topological} can be applied for $2d$ colors.  Notice that the dimension of the
auxiliary space is not reduced!

\section{Topological preliminaries}
\label{sec:topology}

We collect the topological facts used in the proof.  Our topological
conventions follow Hatcher \cite{hatcher2002algebraic}.  All homology and
cohomology groups in this section have coefficients in $\mathbb Q$.
For a finite simplicial complex, we do not distinguish between the
complex and its geometric realization when writing its homology.

A finite simplicial complex on a finite vertex set $V$ is a collection
$K \subset 2^V$ that contains every singleton and is closed under
taking subsets.  For $W \subset V$, the \emph{induced subcomplex} of
$K$ on $W$ is
\[
 K[W] := \braces{\sigma \in K : \sigma \subset W}.
\]

Let $E$ be a topological space, let $V$ be finite, and let
$G_v \subset E$ be non-empty for every $v \in V$.  The \emph{nerve}
of this family is the simplicial complex
\[
 N(G_v : v \in V)
 :=
 \braces{
  W \subset V : \bigcap_{v \in W}G_v \ne \emptyset
 }.
\]
The intersection corresponding to $W = \emptyset$ is understood to be
$E$.  Directly from the definition,
\begin{equation}
 N(G_v : v \in V)[W]
 = N(G_v : v \in W)
 \qquad
 \text{for every } W \subset V.
 \label{eq:induced_nerve}
\end{equation}

An open cover is called a \emph{good cover} if every non-empty
intersection of finitely many members of the cover is contractible.
We shall use the following form of the Nerve Theorem; see
\cite[Corollary~4G.3]{hatcher2002algebraic}.

\begin{prp}[Nerve Theorem]
\label{prp:nerve_theorem}
Let $\mathcal U$ be a good open cover of a paracompact space $X$.
Then $X$ is homotopy equivalent to the nerve of $\mathcal U$.
\end{prp}

In particular, a finite collection of open convex sets is a good cover
of its union.

Let $r \ge 0$ be an integer.  A finite simplicial complex $K$ on $V$
is called \emph{$r$--Leray} if
\[
 \widetilde H_j\parenth{K[W];\mathbb Q} = 0
 \qquad
 \text{for every } W \subset V
 \text{ and every } j \ge r.
\]

If
\[
 V = V_1\mathbin{\dot\cup}\cdots\mathbin{\dot\cup}V_s
\]
is a partition, a subset of $V$ is called \emph{rainbow}
if it contains exactly one vertex from every class.

We use the following partition-matroid consequence of the topological
colorful Helly theorem of Kalai and Meshulam
\cite[Theorem~1.6]{kalai2005topological}.

\begin{prp}[Topological Colorful Helly theorem]
\label{prp:topological_colorful_helly}
Let $K$ be an $r$--Leray complex whose vertex set is partitioned into
$r+1$ non-empty classes
\[
 V = V_1\mathbin{\dot\cup}\cdots\mathbin{\dot\cup}V_{r+1}.
\]
Suppose that every  rainbow set is a face of $K$.  Then $V_i$ is a
face of $K$ for some $i \in \brackets{r+1}$.
\end{prp}

We next recall the form of Alexander duality needed below.  Let $E$ be
an $m$--dimensional affine Euclidean space.  Its one-point
compactification is
\[
 E^+ := E\cup\braces{\infty}.
\]
The neighborhoods of $\infty$ are the sets containing
\[
 \braces{\infty}\cup(E\setminus Q)
\]
for some compact set $Q \subset E$.  The space $E^+$ is homeomorphic
to $S^m$.

\begin{prp}[Alexander duality]
\label{prp:alexander_duality}
Let $B$ be a non-empty proper compact subset of $E^+$.  Then
\[
 \widetilde H_j\parenth{E^+\setminus B;\mathbb Q}
 \simeq
\widetilde{\check H}^{\,m-j-1}\parenth{B;\mathbb Q},
 \qquad
 0 \le j \le m-1.
\]
In particular, if $B$ is connected, then
\[
 \widetilde H_{m-1}
 \parenth{E^+\setminus B;\mathbb Q}
 = 0.
\]
\end{prp}

For locally contractible compact sets this is
\cite[Corollary~3.45]{hatcher2002algebraic}.  The extension to
arbitrary compact sets, obtained by replacing singular cohomology by
\v{C}ech cohomology, is discussed in
\cite[pp.~256--257]{hatcher2002algebraic}.  We use this general form
because the compact set $B$ below is not assumed to be locally
contractible.

The next lemma makes precise the improvement over the usual
$m$--Leray bound.  The ambient dimension alone gives homology
vanishing in degrees $j \ge m$.  The absence of bounded complementary
components allows Alexander duality to eliminate the remaining group
in degree $m - 1$.

\begin{lem}[A finite union without bounded holes]
\label{lem:bounded_hole_nerve}
Let $E$ be an $m$--dimensional affine Euclidean space, where $m \ge 1$,
and let $V$ be a non-empty finite set.  For every $v \in V$, let
$G_v \subset E$ be a non-empty open convex set.  If every connected
component of
\[
 E \setminus \bigcup_{v \in V}G_v
\]
is unbounded, then
\[
 \widetilde H_j\parenth{N(G_v : v \in V);\mathbb Q} = 0
 \qquad
 \text{for every } j \ge m - 1.
\]
\end{lem}

\begin{proof}
Put
\[
 U := \bigcup_{v \in V}G_v
 \qquad\text{and}\qquad
 K := N(G_v : v \in V).
\]
The sets $G_v$, $v \in V$, form a good open cover of $U$.  Since $U$
is metrizable and hence paracompact,
\Href{Proposition}{prp:nerve_theorem} gives
\begin{equation}
 \widetilde H_j\parenth{K;\mathbb Q}
 \simeq
 \widetilde H_j\parenth{U;\mathbb Q}
 \qquad
 \text{for every } j.
 \label{eq:nerve_union_homology}
\end{equation}

Consider the compact set
\[
 B := E^+\setminus U
 = (E\setminus U)\cup\braces{\infty}.
\]
It is non-empty and proper.  We claim that $B$ is connected.  This is
clear if $E\setminus U$ is empty.  Otherwise, let $A$ run through the
connected components of $E\setminus U$.  Each $A$ is closed in $E$ and
unbounded, so its closure in $E^+$ is
$A\cup\braces{\infty}$.
 This closure is connected because $A$ is
connected.
  These sets all contain $\infty$, and their
union is $B$.  Hence, $B$ is connected.

By \Href{Proposition}{prp:alexander_duality},
\[
 \widetilde H_{m-1}\parenth{U;\mathbb Q}
 \simeq
 \widetilde{\check H}^{\,0}\parenth{B;\mathbb Q}
 = 0.
\]
Every connected component of $U$ is a non-compact $m$--manifold, so
\cite[Proposition~3.29]{hatcher2002algebraic}, applied componentwise,
gives
\[
 \widetilde H_j\parenth{U;\mathbb Q} = 0,
 \qquad
 j \ge m.
\]
The result follows from \eqref{eq:nerve_union_homology}.
\end{proof}

The geometric consequence used in the proof of the segment theorem is
the following.

\begin{lem}[Colorful-hole lemma]
\label{lem:colorful_hole}
Let $E$ be an $m$--dimensional affine Euclidean space, where $m \ge 1$,
and let the finite set $V$ be partitioned into non-empty sets
\[
 V = V_1\mathbin{\dot\cup}\cdots\mathbin{\dot\cup}V_m.
\]
For every $v \in V$, let $G_v \subset E$ be a non-empty open convex
set.  Suppose that
\begin{enumerate}[(i)]
\item for every $W \subset V$, every connected component of
\[
 E \setminus \bigcup_{v \in W}G_v
\]
is unbounded;
\item for every $i \in \brackets{m}$,
\[
 \bigcap_{v \in V_i}G_v = \emptyset.
\]
\end{enumerate}
Then there are $v_i \in V_i$, $i \in \brackets{m}$, such that
\[
 \bigcap_{i \in \brackets{m}}G_{v_i} = \emptyset.
\]
\end{lem}

\begin{proof}
Put $K := N(G_v : v \in V)$.  Let $W \subset V$.  If $W$ is non-empty,
assumption \emph{(i)}, \Href{Lemma}{lem:bounded_hole_nerve}, and
\eqref{eq:induced_nerve} give
\[
 \widetilde H_j\parenth{K[W];\mathbb Q} = 0
 \qquad
 \text{for every } j \ge m - 1.
\]
The same assertion is immediate for $W = \emptyset$.  Thus, $K$ is
$(m - 1)$--Leray.  If every  rainbow intersection were non-empty,
every  rainbow set would be a face of $K$.  Applying
\Href{Proposition}{prp:topological_colorful_helly} with $r = m - 1$, we
would obtain an $i \in \brackets{m}$ such that $V_i$ is a face of $K$.
By the definition of the nerve, this means that
\[
 \bigcap_{v \in V_i}G_v \ne \emptyset,
\]
contrary to assumption \emph{(ii)}.
\end{proof}

It remains to construct the family to which
\Href{Lemma}{lem:colorful_hole} will be applied.  The next section gives
the exact intersection criterion, and \Href{Section}{sec:escape} proves
the required condition on complementary components.

\section{Polar encoding of translates of the segment}
\label{sec:encoding}

For non-empty compact convex sets $P, C \subset \R^d$, we first
characterize the existence of $c \in C$ such that
\[
 c + \brackets{-u, u} \subset P.
\]
For a non-empty subset $S$ of a finite-dimensional Euclidean space,
its \emph{support function} is
\[
 h_S(w) := \sup_{x \in S}\iprod{w}{x}.
\]

For two non-empty sets $P, C \subset \R^d$, put
\[
 D_C(P) := \braces{(x - c, y - c) : x, y \in P, c \in C}
 \subset \R^d \times \R^d
\]
and $v := (u, -u)$.

The set $D_C(P)$ provides a useful characterization of the desired
inclusion.
\begin{lem}
\label{lem:product_encoding}
Let $P, C \subset \R^d$ be non-empty compact convex sets.  The set
$D_C(P)$ is compact and convex, and, for every $a, b \in \R^d$,
\begin{equation}
 h_{D_C(P)}(a, b) = h_P(a) + h_P(b) + h_C(-a - b).
 \label{eq:support_difference}
\end{equation}
Moreover,
\begin{equation}
 \text{there is }c \in C\text{ with }
 c + \brackets{-u, u} \subset P
 \quad\Longleftrightarrow\quad
 D_C(P) \cap \braces{\lambda v : \lambda \ge 1} \ne \emptyset.
 \label{eq:ray}
\end{equation}
\end{lem}

\begin{proof}
The first assertion follows because $D_C(P)$ is the image of
$P \times P \times C$ under the linear map
\[
 T : \R^d \times \R^d \times \R^d \longrightarrow \R^d \times \R^d,
 \qquad T(x,y,c) := (x-c,y-c).
\]
Furthermore,
\[
 h_{D_C(P)}(a, b)
 = \max_{x, y \in P, c \in C}
   (\iprod{a}{x - c} + \iprod{b}{y - c})
 = h_P(a) + h_P(b) + h_C(-a - b),
\]
which proves \eqref{eq:support_difference}.

Finally, $(\lambda u, -\lambda u) \in D_C(P)$ if and only if, for
some $c \in C$, both $c + \lambda u$ and $c - \lambda u$ belong to $P$.
If $\lambda \ge 1$, the convexity of $P$ then gives
$c + \brackets{-u, u} \subset P$.  The converse follows by taking
$\lambda = 1$.  This proves \eqref{eq:ray}.
\end{proof}

To rewrite this inclusion in terms of dual inequalities, define the
$2d$--dimensional affine space
\[
 E_u := \braces{(p, q, t) \in \R^d \times \R^d \times \R :
                    \iprod{q}{u} = 1}.
\]

We shall use the following functions.  They are allowed to take the
value $+\infty$.

\begin{dfn}
\label{dfn:dual_function}
For non-empty sets $P, C \subset \R^d$, vectors $p, q \in \R^d$, and
$t \in \R$, define
\[
 \Phi_P^C(p, q, t)
 := h_C(-p) + \sup_{x \in P}
 \braces{\iprod{p}{x} + \abs{\iprod{q}{x} - t}}.
\]
For $x \in \R^d$, its one-point specialization is
\[
 \varphi_x^C(p, q, t)
 := \Phi_{\braces{x}}^C(p, q, t)
 = h_C(-p) + \iprod{p}{x} + \abs{\iprod{q}{x} - t}.
\]
If $P$ is bounded, also put
\begin{equation}
 t_P(p, q) := \frac{h_P(p + q)-h_P(p - q)}{2}.
 \label{eq:explicit_t}
\end{equation}
\end{dfn}

For a non-empty set $C \subset \R^d$ and $x \in \R^d$, define
\begin{equation}
 F_x^C
 := \braces{(p, q, t) \in E_u : \varphi_x^C(p, q, t) < 1}.
 \label{eq:point_dual_set}
\end{equation}

The following lemma records the elementary properties of $\Phi_P^C$
used in the separation argument.

\begin{lem}
\label{lem:dual_function}
Let $P, C \subset \R^d$ be non-empty compact sets.  The function
$\Phi_P^C$ is finite, continuous, and convex.  Moreover,
\begin{equation}
 \Phi_P^C(p, q, t)
 = h_C(-p) +
 \max\braces{h_P(p + q)-t, h_P(p - q)+t}
 \label{eq:Phi_support}
\end{equation}
for all $p, q \in \R^d$ and $t \in \R$.  The function
$t \longmapsto \Phi_P^C(p, q, t)$ has the unique minimizer
$t_P(p, q)$.  If
\[
 a := \frac{p + q}{2},
 \qquad
 b := \frac{p - q}{2},
\]
then
\begin{equation}
 h_{D_C(P)}(a, b)
 = \Phi_P^C(p, q, t_P(p, q))
 \le \Phi_P^C(p, q, t).
 \label{eq:Phi_geometric}
\end{equation}
\end{lem}

\begin{proof}
Using $\abs{s} = \max\braces{s, -s}$ and interchanging the supremum
over $P$ with the maximum over the two signs gives
\eqref{eq:Phi_support}.  The right-hand side is finite, continuous, and
convex.

The maximum of the two affine functions $A - t$ and $B + t$ is uniquely
minimized when $A - t = B + t$.  With $A = h_P(p + q)$ and
$B = h_P(p - q)$, this gives \eqref{eq:explicit_t} and
\[
 \Phi_P^C(p, q, t_P(p, q))
 = h_C(-p) + \frac{h_P(p + q)+h_P(p - q)}{2}.
\]
By the definition of $D_C(P)$,
\[
 h_{D_C(P)}(a, b)
 = h_P(a) + h_P(b) + h_C(-a - b).
\]
Since $a + b = p$, this proves \eqref{eq:Phi_geometric}.
\end{proof}

The one-point sets in \eqref{eq:point_dual_set} have the properties
required in the topological argument.

\begin{lem}
\label{lem:dual_sets}
Let $C \subset \R^d$ be a non-empty compact set.  For every
$x \in \R^d$, the set $F_x^C$ is non-empty, open, and convex in $E_u$.
\end{lem}

\begin{proof}
The function $\varphi_x^C = \Phi_{\braces{x}}^C$ is finite, continuous,
and convex by \Href{Lemma}{lem:dual_function}.  Hence, $F_x^C$ is open
and convex.  Since $u$ is a unit vector,
\[
 \parenth{0, u, \iprod{u}{x}} \in E_u
 \qquad\text{and}\qquad
 \varphi_x^C(0, u, \iprod{u}{x}) = 0.
\]
Thus, $F_x^C$ is non-empty.
\end{proof}

The resulting finite criterion is the following.

\begin{cor}[Finite duality]
\label{cor:finite_duality}
Let $I \subset \R^d$ be a non-empty finite set, and let
$C \subset \R^d$ be non-empty, compact, and convex.  Then
\[
 \text{there is }c \in C\text{ with }
 c + \brackets{-u, u} \subset \conv I
 \quad\Longleftrightarrow\quad
 \bigcap_{x \in I}F_x^C = \emptyset.
\]
\end{cor}

\begin{proof}
Put $P := \conv I$.  For fixed $p, q, t$, the function
\[
 x \longmapsto \iprod{p}{x} + \abs{\iprod{q}{x} - t}
\]
is convex.  Therefore,
\begin{equation}
 \Phi_P^C(p, q, t)
 = \max_{x \in I}\varphi_x^C(p, q, t),
 \qquad
 \bigcap_{x \in I}F_x^C
 = \braces{(p, q, t) \in E_u : \Phi_P^C(p, q, t) < 1}.
 \label{eq:finite_intersection}
\end{equation}

We prove the equivalent assertion that the required translate does not
exist if and only if the intersection in \eqref{eq:finite_intersection}
is non-empty.  Suppose first that the translate does not exist.  By
\Href{Lemma}{lem:product_encoding}, the compact convex set $D_C(P)$ is
disjoint from the ray $\braces{\lambda v : \lambda \ge 1}$.  Strict
separation gives $g = (a, b)$ such that
\[
 h_{D_C(P)}(g)
 < \inf_{\lambda \ge 1}\lambda\iprod{g}{v}.
\]
Thus, $\iprod{g}{v} \ge 0$.  If $\iprod{g}{v} = 0$, then
$h_{D_C(P)}(g) < 0$.  Replacing $g$ by $g + \delta v$ for a sufficiently
small $\delta > 0$, and continuing to denote the new vector by $g$,
makes $\iprod{g}{v}$ positive while preserving
\[
 h_{D_C(P)}(g) < \iprod{g}{v}.
\]
After a  rescaling, we may therefore assume that
\[
 \iprod{a - b}{u} = 1,
 \qquad
 h_{D_C(P)}(a, b) < 1.
\]
Set $p := a + b$ and $q := a - b$.  Then $\iprod{q}{u} = 1$, and
\Href{Lemma}{lem:dual_function}, with the explicit choice
$t = t_P(p, q)$, gives
\[
 \Phi_P^C(p, q, t_P(p, q))
 = h_{D_C(P)}(a, b) < 1.
\]
Thus, the intersection in \eqref{eq:finite_intersection} is non-empty.

Conversely, suppose that $(p, q, t)$ belongs to this intersection, and
put
\[
 a := \frac{p + q}{2},
 \qquad
 b := \frac{p - q}{2}.
\]
By \Href{Lemma}{lem:dual_function},
\[
 h_{D_C(P)}(a, b)
 = \Phi_P^C(p, q, t_P(p, q))
 \le \Phi_P^C(p, q, t) < 1.
\]
On the other hand, $\iprod{(a, b)}{v} = \iprod{q}{u} = 1$, and hence
\[
 \iprod{(a, b)}{\lambda v} = \lambda \ge 1
 \qquad
 \text{for every } \lambda \ge 1.
\]
Thus, $D_C(P)$ is disjoint from the ray.  By
\Href{Lemma}{lem:product_encoding}, the required translate does not
exist.
\end{proof}

\section{The escape property}
\label{sec:escape}
To apply \Href{Lemma}{lem:colorful_hole}, we need the following escape
property of the sets $F_x^C$.

\begin{lem}[Escape lemma]
\label{lem:moving_escape}
Let $I \subset \R^d$ be finite, and let $C \subset \R^d$ be a
non-empty compact set.  Then every connected component of
\[
 E_u \setminus \bigcup_{x \in I}F_x^C
 = \braces{(p, q, t) \in E_u :
  \varphi_x^C(p, q, t) \ge 1 \text{ for every } x \in I}
\]
is unbounded.
\end{lem}

We will write $z = (p, q, t)$ for a point of $E_u$.

The proof consists of two steps.  We first treat a fixed center $c$;
in this case the relevant complements are finite unions of polyhedra.
We then choose, for each $z$, a supporting point $c_z$ of $C$ and
reduce the general case to the fixed-center one.

Before proving the escape lemma, we isolate the polyhedral Farkas-like alternative
used in the first step.

\begin{dfn}
\label{dfn:Q}
For a non-zero vector $v \in \R^n$, let
\[
 H_v := \braces{a \in \R^n : \iprod{a}{v} = 1}.
\]
If $K \subset \R^n$ is non-empty, define
\[
 Q(K, v)
 := \braces{a \in H_v : \iprod{a}{y} \ge 1
  \text{ for every } y \in K}.
\]
\end{dfn}

\begin{lem}[Polyhedral Farkas alternative]
\label{lem:farkas_alternative}
Let $K \subset \R^n$ be a non-empty polytope, and let
$v \in \R^n$ be non-zero.  Then the following statements hold:
\begin{enumerate}[(a)]
\item $Q(K, v) = \emptyset$ if and only if
$K \cap \braces{\lambda v : \lambda < 1} \ne \emptyset$;
\item if $Q(K, v)$ is non-empty and bounded, then
$K \cap \braces{\lambda v : \lambda \ge 1} \ne \emptyset$.
\end{enumerate}
\end{lem}

\begin{proof}
Write $K = \conv\braces{z_1, \dots, z_N}$.  For this fixed choice of
vertices, define the affine evaluation map
\[
 T : H_v \longrightarrow \R^N,
 \qquad
 T(a) := \parenth{\iprod{a}{z_i} - 1}_{i \in \brackets{N}}.
\]
Also, let
\[
 \R_+^N := \braces{y \in \R^N : y_i \ge 0
                    \text{ for every } i \in \brackets{N}}.
\]
Since a linear functional attains its minimum over $K$ at a vertex,
\[
 Q(K, v) = T^{-1} \! \parenth{\R_{+}^{N}}.
\]

Put
\[
 a_0 := \frac{v}{\enorm{v}^2} \in H_v
\]
and define the linear map
\[
 L : v^\perp \longrightarrow \R^N,
 \qquad
 L(h) := \parenth{\iprod{h}{z_i}}_{i \in \brackets{N}}.
\]
Clearly, every $a \in H_v$
has a unique representation $a = a_0 + h$, $h \in v^\perp$, and
\[
 T(a_0 + h) = T(a_0) + L(h).
\]
Thus, $Q(K, v) = \emptyset$ means precisely that the system
\[
 T(a_0) + L(h) \in \R_+^N,
 \qquad h \in v^\perp,
\]
is infeasible.  After choosing a basis in $v^\perp$, the
 Farkas alternative
\cite[Section~5.8.3, p.~263]{boyd2004convex} gives a vector
$\theta = \parenth{\theta_i}_{i \in \brackets{N}}
\in \R_+^N \setminus \braces{0}$ such that
\begin{equation}
 \iprod{\theta}{L(h)} = 0
 \quad\text{for every } h \in v^\perp,
 \qquad
 \iprod{\theta}{T(a_0)} < 0.
 \label{eq:farkas_certificate}
\end{equation}
The first relation in \eqref{eq:farkas_certificate} says that
\[
 \sum_{i \in \brackets{N}}\theta_i z_i = \gamma v
\]
for some scalar $\gamma$.  Since $a_0 \in H_v$, the second relation
gives
\[
 \gamma - \sum_{i \in \brackets{N}}\theta_i
 = \iprod{a_0}{\sum_{i \in \brackets{N}}\theta_i z_i}
   - \sum_{i \in \brackets{N}}\theta_i
 < 0.
\]
Set $s := \sum_{i \in \brackets{N}}\theta_i$.  Then $s > 0$ and
\[
 \frac{\gamma}{s}v
 = \sum_{i \in \brackets{N}}\frac{\theta_i}{s}z_i \in K,
 \qquad
 \frac{\gamma}{s} < 1.
\]
Conversely, suppose that $\lambda v \in K$ for some $\lambda < 1$.
If $a \in Q(K, v)$, then taking the inner product with $a$ in a convex
representation of $\lambda v$ by the vertices of $K$ gives
$\lambda \ge 1$, a contradiction.  This proves \emph{(a)}.

Suppose now that $Q(K, v)$ is non-empty and bounded.  Its recession
cone is
\[
 \braces{h \in v^\perp : \iprod{h}{z_i} \ge 0
                     \text{ for every } i \in \brackets{N}}.
\]
Let $\bar z_i$ be the orthogonal projection of $z_i$ onto $v^\perp$.
The recession cone is trivial because $Q(K, v)$ is bounded.  If $0$
did not lie in the interior, taken in $v^\perp$, of the convex hull of
the $\bar z_i$, the separating hyperplane theorem in $v^\perp$ would
give a non-zero $h \in v^\perp$ such that
\[
 \iprod{h}{\bar z_i} \ge 0
 \qquad\text{for every } i \in \brackets{N}.
\]
That is, it would be a non-zero vector in the recession cone, a contradiction.
Thus, $0$ lies in the interior of
$\conv\braces{\bar z_1, \dots, \bar z_N}$ in $v^\perp$.
Consequently, there are $\theta_i \ge 0$ with
$\sum_{i \in \brackets{N}}\theta_i = 1$ and a scalar $\lambda$ such that
\[
 \sum_{i \in \brackets{N}}\theta_i z_i = \lambda v.
\]
For $a \in Q(K, v)$, taking the inner product with $a$ gives
\[
 \lambda
 = \sum_{i \in \brackets{N}}\theta_i\iprod{a}{z_i} \ge 1.
\]
This proves \emph{(b)}.
\end{proof}

We next introduce the polyhedra that arise when the center is fixed.
For $x, c \in \R^d$, define the affine functions on $E_u$
\[
 \ell_{x,c}^+(p, q, t)
 := \iprod{p}{x - c} + \iprod{q}{x} - t,
 \qquad
 \ell_{x,c}^-(p, q, t)
 := \iprod{p}{x - c} - \iprod{q}{x} + t.
\]
Thus,
\[
 \varphi_x^{\braces{c}}(z)
 = \max\braces{\ell_{x,c}^+(z), \ell_{x,c}^-(z)}.
\]
For finite sets $A, B \subset \R^d$, define the sign polyhedron
\[
 \mathcal P_c(A, B)
 := \braces{z \in E_u :
 \ell_{x,c}^+(z) \ge 1 \text{ for every } x \in A,
 \ \ell_{x,c}^-(z) \ge 1 \text{ for every } x \in B}.
\]
The sets $A$ and $B$ are allowed to intersect.  Finally, put
\[
 J_+^c(A) := \braces{\lambda \in \R :
 c + \lambda u \in \conv A},
 \qquad
 J_-^c(B) := \braces{\lambda \in \R :
 c - \lambda u \in \conv B},
\]
where $\conv\emptyset := \emptyset$.  Each of the sets $J_+^c(A)$ and
$J_-^c(B)$ is either empty or a compact interval.

\begin{lem}[Sign-polyhedron alternative]
\label{lem:sign_polyhedron}
Let $A, B \subset \R^d$ be finite, and let $c \in \R^d$.  Then
\begin{equation}
 \mathcal P_c(A, B) = \emptyset
 \quad\Longleftrightarrow\quad
 J_+^c(A) \cap J_-^c(B) \cap (-\infty, 1) \ne \emptyset,
 \label{eq:empty_alternative}
\end{equation}
Moreover,
\begin{equation}
 \mathcal P_c(A, B) \text{ is non-empty and bounded}
 \quad\Longrightarrow\quad
 J_+^c(A) \cap J_-^c(B) \cap [1, \infty) \ne \emptyset.
 \label{eq:bounded_alternative}
\end{equation}
If $A = \emptyset$ or $B = \emptyset$, then
$\mathcal P_c(A, B)$ is non-empty and unbounded.
\end{lem}

\begin{proof}
Suppose first that $A$ and $B$ are non-empty.  Put
\[
 v := (u, -u) \in \R^d \times \R^d.
\]
In the notation of \Href{Definition}{dfn:Q},
\[
 H_v
 = \braces{(a, b) \in \R^d \times \R^d :
             \iprod{a - b}{u} = 1}.
\]
Define the affine map
\[
 \pi : E_u \longrightarrow H_v,
 \qquad
 \pi(p, q, t)
 := \parenth{\frac{p + q}{2}, \frac{p - q}{2}}.
\]
It is onto: for $(a, b) \in H_v$, take
$p = a + b$, $q = a - b$, and any $t \in \R$.

Let $\pi(p, q, t) = (a, b)$ and set
\[
 \tau := \frac{t - \iprod{q}{c}}{2}.
\]
Direct substitution gives, for every $x \in \R^d$,
\[
 \ell_{x,c}^+(p, q, t)
 = 2\parenth{\iprod{a}{x - c} - \tau},
 \qquad
 \ell_{x,c}^-(p, q, t)
 = 2\parenth{\iprod{b}{x - c} + \tau}.
\]
For fixed $(a, b) \in H_v$, a suitable value of $t$, or equivalently
of $\tau$, exists if and only if
\[
 \frac{1}{2} - \min_{y \in B}\iprod{b}{y - c}
 \le \tau
 \le \min_{x \in A}\iprod{a}{x - c} - \frac{1}{2}.
\]
Hence,
\[
 \pi(\mathcal P_c(A, B))
  = \braces{(a, b) \in H_v :
  \min_{x \in A}\iprod{a}{x - c}
  + \min_{y \in B}\iprod{b}{y - c} \ge 1}
 = Q(K_{A,B}, v),
\]
where
\[
 K_{A,B} := (\conv A - c) \times (\conv B - c).
\]
Furthermore,
\[
 \lambda v \in K_{A,B}
 \quad\Longleftrightarrow\quad
 \lambda \in J_+^c(A) \cap J_-^c(B).
\]
Therefore, \eqref{eq:empty_alternative} follows from part \emph{(a)} of
\Href{Lemma}{lem:farkas_alternative}.  If $\mathcal P_c(A, B)$ is
bounded, then its image under $\pi$ is bounded, and part \emph{(b)}
gives \eqref{eq:bounded_alternative}.

Suppose now that $A = \emptyset$.  Taking $p = 0$, $q = u$, and $t$
sufficiently large gives a point of $\mathcal P_c(A, B)$; increasing
$t$ produces a ray contained in this polyhedron.  If $B = \emptyset$,
the same argument uses $t$ sufficiently negative and then decreases
$t$.  In either case, one of $J_+^c(A)$ and $J_-^c(B)$ is empty, so
\eqref{eq:empty_alternative} remains valid, while
\eqref{eq:bounded_alternative} is vacuous.
\end{proof}

We can now prove the fixed-center form of the escape lemma.

\begin{lem}[Fixed-center escape]
\label{lem:fixed_escape}
Let $I \subset \R^d$ be finite, and let $c \in \R^d$.  Then every
connected component of
\[
 E_u \setminus \bigcup_{x \in I}F_x^{\braces{c}}
 = \braces{z \in E_u :
  \varphi_x^{\braces{c}}(z) \ge 1
  \text{ for every } x \in I}
\]
is unbounded.
\end{lem}

\begin{proof}
The complement in the statement has the finite decomposition
\begin{equation}
 E_u \setminus \bigcup_{x \in I}F_x^{\braces{c}}
 = \bigcup_{A \sqcup B = I}\mathcal P_c(A, B),
 \label{eq:sign_decomposition}
\end{equation}
where the union is taken over all ordered partitions of $I$.  Indeed,
for each $x \in I$, one chooses a sign for which
$\ell_{x,c}^{\pm}(z) \ge 1$.

Fix
\[
 z \in E_u \setminus \bigcup_{x \in I}F_x^{\braces{c}}.
\]
Choose a sign map $\varepsilon_z : I \longrightarrow \braces{+, -}$
such that
\[
 \ell_{x,c}^{\varepsilon_z(x)}(z) \ge 1
 \qquad\text{for every } x \in I.
\]
If both signs are possible, choose either one, and put
\[
 A := \braces{x \in I : \varepsilon_z(x) = +},
 \qquad
 B := \braces{x \in I : \varepsilon_z(x) = -}.
\]
Then $I = A \sqcup B$ and $z \in \mathcal P_c(A, B)$.  By
\eqref{eq:sign_decomposition}, this polyhedron is contained in the
required complement.  Write $\mathcal P := \mathcal P_c(A, B)$.
If $\mathcal P$ is unbounded, its recession cone contains a non-zero
vector.  Therefore, the ray starting at $z$ in this recession direction
lies in the required complement.

It remains to consider the case in which $\mathcal P$ is bounded.
Then both $A$ and $B$ are non-empty.  By
\Href{Lemma}{lem:sign_polyhedron}, specifically by
\eqref{eq:empty_alternative} and \eqref{eq:bounded_alternative}, the
intervals $J_+^c(A)$ and $J_-^c(B)$ have a common point in
$[1, \infty)$, but no common point in $(-\infty, 1)$.  Since they are
intervals,
\begin{equation}
 J_+^c(A) \subset \braces{\lambda \in \R : \lambda \ge 1}
 \quad\text{or}\quad
 J_-^c(B) \subset \braces{\lambda \in \R : \lambda \ge 1}.
 \label{eq:safe_side}
\end{equation}

Assume first that the first inclusion in \eqref{eq:safe_side} holds.
Then \eqref{eq:empty_alternative} in
\Href{Lemma}{lem:sign_polyhedron}, applied to $A$ and $I$, shows that
$\mathcal P_c(A, I)$ is non-empty.  Choose
$w \in \mathcal P_c(A, I)$.  Since
\[
 \mathcal P_c(A, I)
 = \mathcal P_c(A, B)
 \cap \mathcal P_c(\emptyset, I),
\]
the segment $\brackets{z, w}$ lies in $\mathcal P_c(A, B)$, whereas
the ray
\[
 w + s(0, 0, 1),
 \qquad s \ge 0,
\]
lies in $\mathcal P_c(\emptyset, I)$.  Both polyhedra are contained in
$E_u \setminus \bigcup_{x \in I}F_x^{\braces{c}}$.  If the second
inclusion in \eqref{eq:safe_side} holds, choose
$w \in \mathcal P_c(I, B)$ and use
\[
 \mathcal P_c(I, B)
 = \mathcal P_c(A, B)
 \cap \mathcal P_c(I, \emptyset)
\]
together with the ray $w + s(0, 0, -1)$, $s \ge 0$.  Thus, every point of
$E_u \setminus \bigcup_{x \in I}F_x^{\braces{c}}$ can be joined within
this set to an unbounded ray.  Therefore, every connected component is
unbounded.
\end{proof}

The moving-center case follows by choosing a supporting point of $C$.

\begin{proof}[Proof of \Href{Lemma}{lem:moving_escape}]
For every $p \in \R^d$ and $c \in C$,
\[
 h_C(-p) \ge -\iprod{p}{c}.
\]
Consequently,
\[
 F_x^C \subset F_x^{\braces{c}}
 \qquad
 \text{for every } x \in \R^d.
\]
Hence,
\begin{equation}
 E_u \setminus \bigcup_{x \in I}F_x^{\braces{c}}
 \subset E_u \setminus \bigcup_{x \in I}F_x^C.
 \label{eq:monotonicity}
\end{equation}
Take
\[
 z = (p, q, t) \in E_u \setminus \bigcup_{x \in I}F_x^C.
\]
By compactness, we can choose
\[
 c_z \in \operatorname*{argmin}_{c \in C}\iprod{p}{c}.
\]
At $z$ we have
\[
 h_C(-p) = -\iprod{p}{c_z},
\]
and hence
\[
 \varphi_x^C(z) = \varphi_x^{\braces{c_z}}(z)
 \qquad
 \text{for every } x \in I.
\]
Thus,
\[
 z \in E_u \setminus \bigcup_{x \in I}F_x^{\braces{c_z}}.
\]
By \Href{Lemma}{lem:fixed_escape}, the connected component of this set
containing $z$ is unbounded.  By \eqref{eq:monotonicity}, it is a
connected subset of $E_u \setminus \bigcup_{x \in I}F_x^C$.  Hence, the
connected component of the latter set containing $z$ is unbounded.
\end{proof}

\section{Proof of the segment theorem}
\label{sec:segment_proof}

\begin{proof}[Proof of \Href{Theorem}{thm:segment}]
For each $i \in \brackets{2d}$, choose a finite subset $Y_i \subset X_i$
whose convex hull contains both $c_i-u$ and $c_i+u$. Such a set exists
by the definition of the convex hull. Replacing $X_i$ by $Y_i$, we may
therefore assume that every color class is finite.

Put
\[
 C = \conv\braces{c_1, \dots, c_{2d}}.
\]

Let
\[
 V := \braces{(i, x) : i \in \brackets{2d},\ x \in X_i},
 \qquad
 V_i := \braces{(i, x) : x \in X_i}.
\]
The sets $V_i$, $i \in \brackets{2d}$, form a partition of $V$ into
non-empty classes.
For $\xi = (i, x) \in V$, set $G_\xi := F_x^C$.  By
\Href{Lemma}{lem:dual_sets}, the set $G_\xi$ is non-empty, open, and
convex in $E_u$.
For $W \subset V$, put
\[
 I_W := \braces{x \in \R^d : (i, x) \in W
 \text{ for some } i}.
\]
Then
\[
 E_u \setminus \bigcup_{\xi \in W}G_\xi
 = E_u \setminus \bigcup_{x \in I_W}F_x^C.
\]
By \Href{Lemma}{lem:moving_escape}, every connected component of this
set is unbounded.  Since $\dim E_u = 2d$, condition \emph{(i)} of
\Href{Lemma}{lem:colorful_hole} holds with $m = 2d$.

For every $i \in \brackets{2d}$, hypothesis
\eqref{eq:segment_hypothesis} gives
\[
 c_i + \brackets{-u, u} \subset \conv X_i,
\]
and $c_i \in C$.  Hence, \Href{Corollary}{cor:finite_duality} yields
\[
 \bigcap_{\xi \in V_i}G_\xi
 = \bigcap_{x \in X_i}F_x^C
 = \emptyset.
\]
Thus, condition \emph{(ii)} of
\Href{Lemma}{lem:colorful_hole} also holds.  Applying that lemma, we
obtain points $x_i \in X_i$ such that
\[
 \bigcap_{i \in \brackets{2d}}G_{(i, x_i)}
 = \bigcap_{i \in \brackets{2d}}F_{x_i}^C
 = \emptyset.
\]
By \Href{Corollary}{cor:finite_duality}, there is a point $c \in C$
satisfying
\[
 c + \brackets{-u, u}
 \subset \conv\braces{x_1, \dots, x_{2d}},
\]
which is \eqref{eq:segment_conclusion}.
\end{proof}

\section{A colorful Helly theorem for cones}
\label{sec:cones}

We now prove \Href{Theorem}{thm:colorful_helly_cones}. We use the
following theorem of B\'ar\'any \cite[Theorem~1.2]{barany2024positive}.
Recall that $m(k,d) = \max\braces{d+1, 2(d - k+1)}$.

\begin{prp}\label{prp:helly_cones}
Fix $k \in \brackets{d-1}$. Assume $\mathcal F$ is a finite family of convex sets
in $\R^d$. If the intersection of any subfamily of at most $m(k,d)$ sets contains
a $k$-dimensional cone, then $\bigcap_{F \in \mathcal F} F$ contains a
$k$-dimensional cone.
\end{prp}

We will use the classical Colorful Helly theorem of Lov\'asz; see
\cite{barany1982generalization}.
\begin{prp}[Colorful Helly theorem]\label{prp:colorful_Helly}
Let $\mathcal G_1, \dots, \mathcal G_{n+1}$ be finite families of convex sets in
an $n$-dimensional affine space. If the intersection of every rainbow selection is
non-empty, then the intersection of all sets in some family $\mathcal G_i$ is
non-empty.
\end{prp}

We first recall two elementary facts about recession cones. For a non-empty closed
convex set $K \subset \R^d$, its \emph{recession cone} is
\[
\reccone{K} = \braces{v \in \R^d \st z+s v \in K \text{ for every }z \in K \text{ and }s \ge 0}.
\]
We write $\ri P$ for the relative interior of a convex set $P$.
If $P$ is non-empty and $p \in \ri P$, then
\begin{equation}\label{eq:relative_interior}
p+\reccone{\overline P} \subset \ri P.
\end{equation}
Indeed, $\ri P = \ri\overline P$, and a sufficiently small relative ball about $p$
lies in $\overline P$. Translating this ball by a vector in $\reccone{\overline
P}$ leaves it in $\overline P$.

Also, if $x_n \in P$, $t_n \to \infty$, and $x_n/t_n \to v$, then $v \in
\reccone{\overline P}$. To see this, fix $z \in \overline P$ and $s \ge 0$. For
all sufficiently large $n$, convexity gives
\[
\parenth{1-\frac{s}{t_n}}z+\frac{s}{t_n}x_n \in \overline P.
\]
Passing to the limit, we obtain $z+s v \in \overline P$.

\begin{proof}[Proof of \Href{Theorem}{thm:colorful_helly_cones}]
By \Href{Proposition}{prp:helly_cones}, the intersection of every full rainbow
selection contains a $k$-dimensional cone. If a color class is empty, the
conclusion is immediate. We therefore assume that all color classes are non-empty.

Put $P_i = \bigcap_{F \in \mathcal F_i} F$, $i \in \brackets{N}$, and assume that
none of the sets $P_i$ contains a $k$-dimensional cone. For every non-empty $P_i$,
define
\[
L_i = \Lin\parenth{\reccone{\overline{P_i}}}.
\]
By \eqref{eq:relative_interior}, $\dim L_i \le k - 1$. Indeed, $k$ linearly
independent vectors in $\reccone{\overline{P_i}}$ would generate a $k$-dimensional
cone whose translate is contained in $P_i$.

Choose a surjective linear map $\pi : \R^d \to \R^{k - 1}$ whose restriction to
each $L_i$ is injective. Such a map exists, since there are only finitely many
subspaces $L_i$, each of dimension at most $k - 1$. For $k = 1$, take the zero map
to $\R^0$.

Thus, two recession vectors of the same set $\overline{P_i}$ must coincide
whenever their images under $\pi$ coincide. We will use the rainbow assumption to
obtain two distinct such vectors in one color class.

For each rainbow selection $R$, choose a $k$-dimensional simplicial cone $C_R$
with apex at the origin and a point $a_R$ such that
\[
a_R+C_R \subset \bigcap_{F \in R} F.
\]
The restriction of $\pi$ to the $k$-dimensional space $\Lin C_R$ has a non-trivial
kernel. Hence, the subspace
\[
W_R = \Lin C_R \cap \ker\pi
\]
has dimension at least one. Choose a linear functional $\ell : \R^d \to \R$ which
does not vanish identically on any $W_R$. To justify this choice, observe that the
functionals vanishing on a fixed $W_R$ form a proper linear subspace of the dual
space. Finitely many proper linear subspaces cannot cover the dual space.

Consequently, the affine hyperplane of $\ker\pi$ given by
\[
H = \braces{z \in \ker\pi \st \ell(z) = 1}
\]
meets every $W_R$. Choose $z_R \in W_R \cap H$. In particular, $\ell(z_R) = 1$, so
all the vectors $z_R$ are non-zero and have the same normalization. Since
$C_R - C_R = \Lin C_R$,
we can write $z_R = u_R - w_R$ with $u_R, w_R \in C_R$. Thus,
\begin{equation}\label{eq:rainbow_directions}
u_R-w_R = z_R,\qquad \pi u_R = \pi w_R,\qquad \ell(u_R-w_R) = 1.
\end{equation}
Choose $M > 0$ satisfying
\[
2\enorm{a_R}+\enorm{u_R}+\enorm{w_R} \le M
\]
for every rainbow selection $R$.

For $t \ge 1$, consider the affine space
\[
E_t = \braces{(x, y) \in \R^d \times \R^d \st \pi x = \pi y,\quad \ell(x-y) = t}.
\]
The geometry of this lift is seen by regarding $(x, y)$ as the endpoints of the
directed segment from $y$ to $x$. The condition $\pi x = \pi y$ means that this
segment is parallel to $\ker\pi$, while $\ell(x-y) = t$ fixes its displacement as
measured by $\ell$. In the midpoint and difference coordinates
\[
c = \frac{x+y}{2},\qquad z = x-y,
\]
the affine space $E_t$ becomes $\R^d \times tH$: the midpoint is free, and the
difference lies in $tH$. Since $\dim H = d - k$, we obtain $\dim E_t = 2d - k = N
- 1$.

For each set $F$ in the original families, define the convex set
\[
G_F(t) = \braces{(x, y) \in E_t \st x, y \in F,\quad \enorm{x}+\enorm{y} \le Mt}.
\]
Thus $G_F(t)$ is a bounded part of the section $(F \times F) \cap E_t$ in
$\R^{2d}$. Its points represent segments contained in $F$, with direction and
normalization prescribed by $tH$. The common norm bound controls the positions of
both endpoints on the same scale as their displacement.

The intersection of every rainbow selection of these sets is non-empty. Indeed,
for the corresponding rainbow $R$, it contains the pair
\[
\parenth{a_R+t u_R,\ a_R+t w_R}.
\]
As $t$ varies, these pairs lie on a ray from the diagonal point $(a_R,a_R)$ in
$\R^{2d}$; their differences are $t z_R \in tH$. The defining linear equalities
follow from \eqref{eq:rainbow_directions}, and the norm bound follows from the
choice of $M$ and the inequality $t \ge 1$.

Applying \Href{Proposition}{prp:colorful_Helly} in $E_t$, we obtain an index $i(t)
\in \brackets{N}$ and points $x_t,y_t \in P_{i(t)}$ satisfying
\begin{equation}\label{eq:monochromatic_pairs}
\pi x_t = \pi y_t,\qquad \ell(x_t-y_t) = t,\qquad \enorm{x_t}+\enorm{y_t} \le Mt.
\end{equation}
Letting $t$ run through the positive integers, we can choose a sequence $t_n \to
\infty$ along which $i(t_n)$ is a fixed index $i$. The last inequality in
\eqref{eq:monochromatic_pairs} allows us to pass to a further subsequence such
that
\[
\frac{x_{t_n}}{t_n} \to v,\qquad \frac{y_{t_n}}{t_n} \to w.
\]
By the recession observation above, $v, w \in \reccone{\overline{P_i}} \subset
L_i$. Dividing the first two equalities in \eqref{eq:monochromatic_pairs} by $t_n$
and passing to the limit, we obtain
\[
\pi v = \pi w,\qquad \ell(v-w) = 1.
\]
Since $\pi$ is injective on $L_i$, the first equality implies $v = w$,
contradicting the second. This completes the proof.
\end{proof}

The common bound $\enorm{x}+\enorm{y} \le Mt$ ensures that the normalized pairs
have convergent subsequences. In particular, the proof applies to arbitrary convex
sets; it uses the closure of the monochromatic intersection only after this
intersection has been shown to be non-empty.

\section*{Acknowledgments}
The author thanks Márton Naszódi, Alexandr Polyanskii, and Roman Karasev for fruitful discussions. The author is deeply grateful to Imre Bárány for his kindness and advice, and for proofreading the first version of the manuscript.

\bibliographystyle{alpha}
\bibliography{../work_current/uvolit}

\begin{thebibliography}{DLLHRS17}

\bibitem[B{\'a}r82]{barany1982generalization}
Imre B{\'a}r{\'a}ny.
\newblock A generalization of {C}arath{\'e}odory's theorem.
\newblock {\em Discrete Mathematics}, 40(2-3):141--152, 1982.

\bibitem[B{\'a}r21]{barany2021combinatorial}
Imre B{\'a}r{\'a}ny.
\newblock {\em Combinatorial convexity}, volume~77.
\newblock American Mathematical Soc., 2021.

\bibitem[B{\'a}r24]{barany2024positive}
Imre B{\'a}r{\'a}ny.
\newblock Positive bases, cones, {H}elly-type theorems.
\newblock {\em Mathematica Slovaca}, 74(3):717--722, 2024.

\bibitem[BKP82]{barany1982quantitative}
Imre B{\'a}r{\'a}ny, Meir Katchalski, and Janos Pach.
\newblock {Q}uantitative {H}elly-type theorems.
\newblock {\em Proceedings of the American Mathematical Society},
  86(1):109--114, 1982.

\bibitem[BV04]{boyd2004convex}
Stephen Boyd and Lieven Vandenberghe.
\newblock {\em Convex optimization}.
\newblock Cambridge university press, 2004.

\bibitem[BW03]{boroczky2003covering}
K{\'a}roly B{\"o}r{\"o}czky and Gergely Wintsche.
\newblock Covering the sphere by equal spherical balls.
\newblock In {\em Discrete and computational geometry}, pages 235--251.
  Springer, 2003.

\bibitem[DFN21]{damasdi2021colorful}
G{\'a}bor Dam{\'a}sdi, Vikt{\'o}ria F{\"o}ldv{\'a}ri, and M{\'a}rton
  Nasz{\'o}di.
\newblock Colorful {H}elly-type theorems for the volume of intersections of
  convex bodies.
\newblock {\em Journal of Combinatorial Theory, Series A}, 178:105361, 2021.

\bibitem[DLLHRS17]{de2017quantitative}
Jes{\'u}s~A. De~Loera, Reuben~N. La~Haye, David Rolnick, and Pablo Sober{\'o}n.
\newblock Quantitative combinatorial geometry for continuous parameters.
\newblock {\em Discrete \& Computational Geometry}, 57(2):318--334, 2017.

\bibitem[Hat02]{hatcher2002algebraic}
Allen Hatcher.
\newblock {\em Algebraic topology}, volume~1.
\newblock Cambridge university press Cambridge, 2002.

\bibitem[IN22]{ivanov2022quantitative}
Grigory Ivanov and M{\'a}rton Nasz{\'o}di.
\newblock A quantitative {H}elly-type theorem: containment in a homothet.
\newblock {\em SIAM Journal on Discrete Mathematics}, 36(2):951--957, 2022.

\bibitem[IN24]{ivanov2024steinitz}
Grigory Ivanov and Márton Naszódi.
\newblock Quantitative {S}teinitz theorem: {A} polynomial bound.
\newblock {\em Bulletin of the London Mathematical Society}, 56(2):796--802,
  2024.

\bibitem[IN26]{Ivanov2026Hellynumbers}
Grigory Ivanov and Márton Naszódi.
\newblock Helly numbers for quantitative {H}elly-type results.
\newblock {\em Journal of Combinatorial Theory, Series A}, 220:106160, May
  2026.

\bibitem[Iva25]{ivanov_QST_polarity_2025}
Grigory Ivanov.
\newblock Quantitative {S}teinitz {T}heorem and {P}olarity.
\newblock {\em Discrete \& Computational Geometry}, May 2025.

\bibitem[Iva26]{Ivanov2026Colorful_pos_bas}
Grigory Ivanov.
\newblock Colorful positive bases decomposition and {H}elly-type results for
  cones.
\newblock {\em Linear Algebra and its Applications}, 732:108–125, March 2026.

\bibitem[Kat78]{katchalski1978helly}
Meir Katchalski.
\newblock A {H}elly type theorem for convex sets.
\newblock {\em Canadian Mathematical Bulletin}, 21(1):121--123, 1978.

\bibitem[KM05]{kalai2005topological}
Gil Kalai and Roy Meshulam.
\newblock A topological colorful {H}elly theorem.
\newblock {\em Advances in Mathematics}, 191(2):305--311, 2005.

\bibitem[Sob16]{soberon2016helly}
Pablo Sober{\'o}n.
\newblock Helly-type theorems for the diameter.
\newblock {\em Bulletin of the London Mathematical Society}, 48(4):577--588,
  2016.

\bibitem[Ste13]{steinitz1913bedingt}
Ernst Steinitz.
\newblock Bedingt konvergente {R}eihen und konvexe {S}ysteme.
\newblock {\em J. Reine Angew. Math.}, 143:128--176, 1913.

\end{thebibliography}

\end{document}